\documentclass[12pt]{amsart}
\usepackage{amsfonts}
\usepackage{amscd}
\usepackage{amssymb}
\usepackage{graphics}
\usepackage{amsmath}
\usepackage{tikz}

\usepackage{hyperref}
\hypersetup{colorlinks,citecolor=blue}

\newtheorem{prop}{Proposition}[section]
\newtheorem{thm}[prop]{Theorem}
\newtheorem{lemma}[prop]{Lemma}

\newtheorem{conj}[prop]{Conjecture}
\newtheorem{hyp}[prop]{Hypothesis}
\theoremstyle{definition}

\newtheorem{rem}[prop]{Remark}

\newtheorem{exam}[prop]{Example}

\newcommand{\cH}{{\mathcal{H}}}
\newcommand{\cA}{{\mathcal{A}}}
\newcommand{\D}{{\mathcal{D}}}
\newcommand{\He}{{\mathcal{H}}}
\newcommand{\cL}{{\mathcal{L}}}
\newcommand{\cR}{{\mathcal{R}}}
\newcommand{\cM}{{\mathcal{M}}}
\newcommand{\cZ}{{\mathcal{Z}}}
\newcommand{\dg}{{\mathrm{deg}}}
\newcommand{\er}{\sim_R}
\newcommand{\el}{\sim_L}
\newcommand{\elr}{\sim_{LR}}

\newcommand{\Z}{{\mathbb{Z}}}

\begin{document}
\title{Cells in Coxeter groups, I}

\author[M. V. Belolipetsky]{Mikhail V. Belolipetsky}
\thanks{Belolipetsky partially supported by EPSRC and CNPq research grants} 
\address{
IMPA\\
Estrada Dona Castorina 110\\
22460-320 Rio de Janeiro, Brazil}
\email{mbel@impa.br}

\author[P. E. Gunnells]{Paul E. Gunnells}
\thanks{Gunnells partially supported by NSF grant DMS 08--01214.}
\address{
Department of Mathematics and Statistics\\
University of Massachusetts\\
Amherst, MA 01003\\
USA}
\email{gunnells@math.umass.edu}



\begin{abstract}
The purpose of this article is to shed new light on the combinatorial
structure of Kazhdan-Lusztig cells in infinite Coxeter groups $W$. Our
main focus is the set $\D$ of distinguished involutions in $W$, which
was introduced by Lusztig in one of his first papers on cells in
affine Weyl groups. We conjecture that the set $\D$ has a simple
recursive structure and can be enumerated algorithmically starting
from the distinguished involutions of finite Coxeter groups. Moreover,
to each element of $\D$ we assign an explicitly defined set of
equivalence relations on $W$ that altogether conjecturally determine
the partition of $W$ into left (right) cells. We are able to prove
these conjectures only in a special case, but even from these partial
results we can deduce some interesting corollaries.
\end{abstract}

\maketitle

\section{Introduction} In their seminal paper \cite{KL} Kazhdan and
Lusztig introduced the notion of cells. These are equivalence classes
in a Coxeter group $W$ and corresponding Hecke algebra $\He$ that can
be defined combinatorially and that have deep connections with
representation theory. Since then there has been considerable interest
in cells and related topics: we can mention the famous papers by
Lusztig on cells in affine Weyl groups, research of J.-Y. Shi on the
combinatorics of cells, the work of Bezrukavnikov and Ostrik on
geometry of unipotent conjugacy classes of simple complex algebraic
groups among the numerous other contributions.  See \cite{G} for
references.  Even a short survey of related results would take us far
beyond the scope of this paper.

The purpose of this article is to shed new light on the combinatorial
structure of cells in infinite Coxeter groups. Our main focus is the
set $\D$ of distinguished involutions in $W$, which was introduced by
Lusztig in \cite{L2}. We conjecture that the set $\D$ has a simple
recursive structure and can be enumerated algorithmically starting
from the distinguished involutions of finite Coxeter groups. Moreover,
to each element of $\D$ we assign an explicitly defined set of equivalence
relations on $W$ that altogether conjecturally determine the partition of
$W$ into left (right) cells. We are able to prove these conjectures only
in a special case, but even from these partial results we can deduce some
interesting corollaries. For example, we show that many non-affine infinite
Coxeter groups contain infinitely many one-sided cells. This was known
before only for a special class of right-angled Coxeter groups
\cite{Bel} and several other hyperbolic examples \cite{Bed}.

In a forthcoming article \cite{BG2} we will present the experimental
support for the conjectures. In particular, we will show that the
conjectures hold for infinite affine groups of small rank (see also
\cite{BG1} for the affine groups of rank $3$) and for large subsets of
hyperbolic triangle groups. Some of the experimental results were
first presented in \cite{G}. The results of the current article and
\cite{BG2} were announced in \cite{BG1}.

The paper is organized as follows. In \S\ref{sec:conj} we introduce
our main conjectures and show that if valid, these conjectures indeed
determine the partition of $W$ into cells (see
Theorem~\ref{prop1}). Section~\ref{sec:res} provides a proof of the
conjectures in a special case when reduced expressions of the elements
of $W$ satisfy certain restrictions. Here we make extensive use of
some unpublished results of Lusztig and Springer. In \S\ref{sec:appl}
we present some corollaries and applications of the results.


\noindent {\em Acknowledgments.} We would like to thank W.~Casselman,
J.~Humphreys, N.~Libedinsky and D.~Rumynin for helpful discussions.
We thank Jian-yi Shi for pointing a gap in the proof of Theorem~\ref{thm1} 
in the published version of the paper and for his help with the correction.

\section{Conjectures}\label{sec:conj}

Let $(W,S)$ be a Coxeter system. We will usually denote general
elements of $W$ by $v, w, x, y, z$, and simple reflections from $S$ by
$s$, $t$. For $x, y \in W$, by $z = x.y$ we mean that $z = xy$ and
$l(z) = l(x) + l(y)$, where $l\colon W\rightarrow \Z$ is the length
function in $(W,S)$.

Throughout the paper we use the terminology of \cite{KL,
L1, L2}.  Thus, let $\cH$ denote the Hecke algebra of $W$ over the
ring $\cA = \Z[q^{1/2}, q^{-1/2}]$ of Laurent polynomials in
$q^{1/2}$. Along with the standard basis $(T_w)_{w\in W}$ of $\cH$ we
have the basis $(C_w)_{w\in W}$ of \cite{KL}, where $C_w = \sum_{y\le
w}(-1)^{l(w)-l(y)}q^{l(w)/2-l(y)} P_{y,w}(q^{-1})T_y$ and
$$P_{y,w} = \mu(y,w)q^{\frac12(l(w)-l(y)-1)} + \mathrm{lower\ degree\
terms}$$ are the so-called {\em Kazhdan-Lusztig polynomials}
introduced in \cite{KL}. Considering the multiplication of the basis
elements in $\cH$ we see that there exist $h_{x,y,z}\in\cA$ such that
$C_x C_y = \sum_z h_{x,y,z} C_z.$ The value of $a(z)$ is defined to be
the smallest integer such that $q^{-\frac{a(z)}{2}}h_{x,y,z} \in
\Z[q^{-\frac12}]$ for all $x, y \in W$, or to be infinity if such an
integer does not exist. If the function $a$ on $W$ takes only finite
values (which is conjecturally true for any group $W$), then for every
$x, y, z \in W$ we can write
$$h_{x,y,z} = \gamma_{x,y,z}q^{\frac{a(z)}2} +
\delta_{x,y,z}q^{\frac{a(z)-1}2} + \mathrm{lower\ degree\ terms}.$$
This formula defines the constants $\gamma_{x,y,z}$ and
$\delta_{x,y,z}$ that we will need later.

Using the polynomials $P_{y,w}$ one can define preorders $\le_L$,
$\le_R$, $\le_{LR}$ and the associated equivalence relations $\el$,
$\er$, $\elr$ on $W$ \cite{KL}. The equivalence classes for
$\el$ (respectively $\er$, $\elr$) are called {\em left cells}
(resp.~{\em right cells}, {\em two-sided cells}) in $W$. Every result
about left cells translates to right cells and vice versa by the
duality, so in our considerations we will usually mention only one of
the two.

Let $\D_i = \{ z\in W \mid l(z) - a(z) - 2\delta(z) = i \}$, where
$l(z)$ is the length of $z$ in $(W,S)$, $\delta(z)$ is the degree of
the polynomial $P_{e,z}$, so $P_{e,z} = \pi(z)q^{\delta(z)} +
\mathrm{lower\ degree\ terms}$, and the function $a(z)$ is defined as
above.  The set $\D = \D_0$ is the set of {\em distinguished
involutions} of $W$, which was introduced in \cite[\S1.3]{L2}.


Our goal is to detect an inductive structure inside $\D$ and to
describe an explicit relationship between the elements of $\D$ and the
equivalence relations on $W$ that determine the partition into
cells. To this end let us formulate two conjectures.


We first introduce some more notations. Let $w\in W$. Denote by $\cZ(w)$
the set of all $v\in W$ such that $w = x.v.y$ for some $x, y\in W$ and
$v\in W_I$ for some $I\subset S$ with $W_I$ finite. We call $v\in \cZ(w)$
{\em maximal in $w$} and write $v\in \cM(w)$, if it is not a proper subword
of any other $v'\in \cZ(w)$ such that $w = x'.v'.y'$ with $x'\le x$ and
$y'\le y$. Let $\cZ = \cZ(W)$ be the union of $\cZ(w)$ over all $w\in W$,
$\D_f := \D\cap \cZ$ be the set of distinguished involutions of the finite
standard parabolic subgroups of $W$ and $\D_f^\bullet = \D_f
\smallsetminus (S\cup\{1\})$. We will call $w = x.v.y$ {\em rigid at
$v$} if (i) $v\in\D_f$, (ii) $v$ is maximal in $w$, and (iii) for
every reduced expression $w = x'.v'.y'$ with $v'\in\cM(w)$ and $a(v')\ge
a(v)$, we have $l(x) = l(x')$ and $l(y) = l(y')$. This notion of {\em
combinatorial rigidity} for the elements of $W$ is essential for our
conjectures and results. We refer to \cite[\S4]{BG1} for some
comments about its meaning.

\begin{conj}\label{conj1} (``distinguished involutions'')
Let $v = x.v_1.x^{-1} \in \D$ with $v_1\in \D_f^\bullet$ and $a(v) = a(v_1)$, and let $v' = s.v.s$ with $s\in S$. Then if $sxv_1$ is rigid at $v_1$, we have $v'\in\D$.
\end{conj}

\begin{conj}\label{conj2} (``basic equivalences'')
Let $w = y.v_0$ with $v_0\in \cM(w)$.
\begin{itemize}
\item[(a)] Let $u = x.v_1.x^{-1}\in\D$ satisfy $a(u) \le a(v_0)$ and
suppose $w' = w.u$ has $a(w') = a(w)$. Assume that there exists
$v_{01}$ such that $v_0 = v_0'.v_{01}$, $v_{01}'xv_1$ is rigid at
$v_1$ for every $v_{01}'$ such that $v_0 = v_0''.v_{01}'$ and
$l(v_{01}') = l(v_{01})$, and such that the right descent sets satisfy
$\cR(w'v_{01}^{-1})\subsetneq\cR(w)$. Then we can choose such $v_{01}$
so that $\mu(w,w'v_{01}^{-1})\neq 0$, and hence $w \er w'v_{01}^{-1}
\er w'$.



\item[(b)] Let $w'' = w.v_1$ with $v_1 \in \D_f$ not maximal in $w''$ and $a(w'') = a(w)$.
Then we can write $w = y.v_{01}.v_{02}.v_{03}$ such that $v_{03}.v_1\in \cM(w'')$,
$\cR(w''v_{02}^{-1})\neq\cR(w)$, and $\mu(w,w''v_{02}^{-1})\neq 0$. So again
$w \er w''v_{02}^{-1} \er w''$.
\end{itemize}
\end{conj}

In practice it is usually easy to find the required endings $v_{01}$
and $v_{02}$ as in the conjecture. This can be seen, in particular, in
the statement of Theorem~\ref{thm2} and in results from \cite{BG2}. At
the same time some examples of \cite{BG2} show that there are cases,
such as affine $\tilde{{\mathrm F}}_4$, that require special
attention.

In order to apply these conjectures we need to recall some other
conjectures from the theory. The first is a variant of a conjecture of
Lusztig about the function $a$ (cf.~\cite[\S13.12]{L3}):
\begin{hyp} (``the function $a$'') \label{conj_a}
For every $w\in W$, $a(w) = a'(w)$ where
\[
a'(w) = \max_{v\in \cM(w)} a(v).
\]
\end{hyp}
One of the immediate corollaries of this hypothesis is that there
exists a constant $N\ge 0$, which depends only on $(W,S)$, such that
for every $w\in W$, $a(w)\le N$. The groups whose $a$-function
satisfies this property are called {\em bounded}. For affine and some
hyperbolic Coxeter groups the boundedness can be verified directly
\cite{L1, Bel}.

We recently learned that there exist Coxeter groups for which
Hypothesis~\ref{conj_a} is false. An example of such group $W$ was presented by
Shi in \cite{Sh2}, who showed that in $W=\tilde{{\mathrm A}}_{10}$ there exists an
explicitly defined element $w$ such that $a(w) > a'(w)$. On the other hand,
in recent papers \cite{Xi2, Zh} it is proved that Hypothesis~\ref{conj_a} holds
for Coxeter groups with a complete graph and for triangle Coxeter groups. It would
be interesting to understand the precise range of the groups for which the function
$a(w)$ satisfies the hypothesis, as well as the extent of its failure for other groups.
We can expect that our main conjectures are valid and the analogue of the following
Theorem~\ref{thm1} holds even without Hypothesis~\ref{conj_a}, but we do not know
how to prove it.


\begin{rem}
There is a small gap in the proof of Theorem 4.2 in \cite{Bel} which,
however, is easy to fix: One has to replace the corresponding part of
line 3 on page 332 there by ``in which $i_{j+1} = i_j+1$, for $j = 1,
\ldots, p-1$, and $s_{i_j}\in \cL(s_{i_1-1}\ldots s_1y)$, for $j = 1,
\ldots, p$'' and everywhere after in this paragraph replace $y$ by $y'
= s_{i_1-1}\ldots s_1y$. We would like to thank Nanhua Xi for pointing
out this issue.
\end{rem}

Finally, we recall the well known positivity conjecture (see e.g. \cite[\S3]{L1}).
\begin{conj}  (``positivity'') \label{conj_pos}
For all $x,y,z \in W$, the coefficients of Kazhdan-Lusztig polynomials
$P_{x,y}(q)$ and polynomials $h_{x,y,z}(q^{1/2},q^{-1/2})$ are
positive integers.
\end{conj}

Positivity of the coefficients of Kazhdan-Lusztig polynomials is well
known for finite and affine Weyl groups where it is proved using the
relation between Kazhdan-Lusztig polynomials and singularities of
Schubert varieties \cite{KL2,L1}. Later on this result was extended to
all crystallographic Coxeter groups using similar geometric ideas
(cf. \cite[Theorem~12.2.9]{Ku}).  This, however, can hardly be
generalized to non-crystallographic cases.  Another approach to
positivity based on categorification of the Hecke algebra was
suggested by Soergel in \cite{So}. We refer to \cite{Li} for some
recent results in this direction.\footnote{\textbf{Note added in the
proofs:} The proof of the positivity conjecture for arbitrary Coxeter
groups using Soergel bimodules was recently announced by Elias and
Williamson, see \texttt{arXiv:1212.0791}.}

\begin{lemma} \label{lem1}
Assume Conjectures~\ref{conj1}, \ref{conj2} and
Hypothesis~\ref{conj_a} hold for $W$.  Let $w = x_1.v_1.x_2.v_2$ with
$v_i\in\cM(w)$, $a(v_1) \ge a(v_2) > a(x_2)$ and $a(x_2v_2) =
a(v_2)$. Then $w\er x_1v_1$.
\end{lemma}
\begin{proof}
We are going to use induction by $a(v_2)$. First note that if $a(v_2)
= 1$, then $v_2 \er s \in S$, $x_2 = e$ and
$\cR(w)\not\supset\cR(x_1v_1)$ (by the maximality of $v_2$ in $w$), so
$w = x_1v_1v_2\er x_1v_1$ follows from the definition of $\er$.

If $a(v_2)>1$ and $v_2\not\in\D_f$, we can replace it by some $v_2'$
so that $w\er x_1.v_1.x_2.v_2'$ and $v_2'\in\D_f$. In order to do so first
find a distinguished involution $v_2'$ of the finite parabolic subgroup $W_I$
containing $v_2$ which is right equivalent to $v_2$ in $W_I$ (its
existence and uniqueness is proved in \cite{L2}). Then using the
maximality of $v_2$ in $w$ and the known properties of the relation $\er$
we can lift $v_2\er v_2'$ in $W_I$ to $x_1.v_1.x_2.v_2 \er x_1.v_1.x_2.v_2'$
in $W$.

Now, assume that $a(v_2)>1$, $v_2\in\D_f$ and $v_{01}'x_2v_2$ is rigid
at $v_2$ for every $v_{01}'$ such that $v_1 = v_1''.v_{01}'$ and
$a(v_{01}') < a(v_2)$.  Then $u = x_2v_2x_2^{-1}$ (and also
$x_1v_1x_2v_2x_2^{-1}$) is reduced, indeed, if it is not then there is
$s\in\cR(x_2)$ such that $sv_2s = v_2$, which contradicts the rigidity
of $x_2v_2$ at $v_2$. Rigidity of $x_2v_2$ at $v_2$ enables us to
apply Conjecture~\ref{conj1} $l(x_2)$ times starting from $v_2$ to
show that $u\in\D$. By Hypothesis~\ref{conj_a}, it follows that $a(u)
= a(v_2)$ and $a(wx_2^{-1}) = a(x_1v_1)$. We are now in a position to
use Conjecture~\ref{conj2}(a), which gives $x_1v_1 \er
x_1v_1x_2v_2x_2^{-1}$.  As $a(x_2)<a(v_2)$, we can apply the
inductions hypothesis together with Conjecture~\ref{conj2}(a) to show
that in turn $x_1v_1x_2v_2x_2^{-1} \er x_1v_1x_2v_2$, which finishes
the proof for this case.

It remains to consider the case when $v_{01}'x_2v_2$ is not rigid at
$v_2$ for some $v_{01}'$ as above. This implies that we can write
$v_{01}'x_2v_2 = x.v.y$ with $v\in\cM(xvy)$, where $a(v) = a(v_2)$,
$a(y) < a(v)$ and $l(y) > 0$. Then the induction hypothesis applies to
show that $v_{01}'x_2v_2 \er xv = v_{01}'x_3$ and $w \er x_1v_1x_3$
with $a(x_3) \le a(v_2)$.  Applying again the induction hypothesis, if
necessary, we get $x_1v_1x_3 \er x_1v_1x_4v_4$ which satisfies either
the assumptions of the lemma or the assumptions of
Conjecture~\ref{conj2}(b).  In the former case we repeat the proof of
the lemma for $w = x_1v_1x_4v_4$ and in the latter apply the
conjecture. As each iteration reduces lengths of the elements, we
eventually arrive to a situation when $v_{01}'x_2v_2$ is rigid at
$v_2$ for every admissible choice of $v_{01}'$, or when $w$ satisfies
the assumptions of Conjecture~\ref{conj2}(b) so that it gives $w \er
x_1v_1$.
\end{proof}

\begin{thm}\label{prop1}
If Conjectures~\ref{conj1}, \ref{conj2}, \ref{conj_pos} and
Hypothesis~\ref{conj_a} are satisfied for an infinite Coxeter group
$W$, then the following are true:
\begin{itemize}
\item[(1)] The set $\D$ of distinguished involutions consists of the union of $v\in\D_f$ and the elements of $W$ obtained from them using Conjecture~\ref{conj1}.
\item[(2)] The relations described in Conjecture~\ref{conj2} determine the partition of $W$ into right cells, i.e. $x\er y$ in $W$ if and only if there exists a sequence $x = x_0$, $x_1$, \ldots, $x_n = y$ in $W$ such that $\{x_{i-1}, x_i\} = \{v,v'\}$ as in~\ref{conj2} for every $i = 1, \ldots , n$.
\item[(3)] The relations described in Conjecture~\ref{conj2} together with its $\el$-analogue determine the partition of $W$ into two-sided cells.
\end{itemize}
\end{thm}

\begin{proof}
By \cite[Theorem 1.10]{L2}, Hypothesis~\ref{conj_a} and Conjecture~\ref{conj_pos}
imply that each right cell of $W$ contains a uniquely
defined distinguished involution. (In his paper Lusztig is primarily
interested in affine Weyl groups, which is also emphasized by the
title the paper, but in the proofs he uses only positivity and
boundedness properties, both of which follow from our assumptions.)
To prove assertions (1) and (2) we thus need to show that for any
$w\in W$ we can find a sequence as in $(2)$ such that $x_0 = w$ and
$x_n = d$ where $d = d(w)$ satisfies the conditions of
Conjecture~\ref{conj1}.

So let $w\in W$ be an arbitrary element. If $\#\cR(w) = 1$ then $w\er
w' = ws$ such that $l(ws) = l(w)-1$. Repeating this procedure we come
to $w\er w_1$ with either $\#\cR(w_1) > 1$ or $w_1 = s \in S$. In the
second case $w_1\in\D$ and we are done. Thus we can assume that
$\#\cR(w_1) > 1$ and $w_1 = x.v$ with a maximal $v$. If
$a(x) \ge a(v)$, then using Hypothesis~\ref{conj_a} we can write $w_1
= x_1.v_1.x_2.v_2$ as in Lemma~\ref{lem1}, which implies $w_1\er
x_1.v_1$. Repeating the procedure if necessary, we come to $w_1\er w_2
= y.u$ with $\#\cR(w_2) > 1$, $u\in\cM(w_2)$ and $a(y) < a(u) =
a(w_2)$. As in the proof of Lemma~\ref{lem1} we can assume here that
$u\in\D_f$. If $w_2$ is rigid at $u$, we can apply Conjecture~\ref{conj1}
$l(y)$ times to show that $y.u.y^{-1}\in\D$. Otherwise we can shift
$u$ to the left using Conjecture~\ref{conj2}(b), and repeating this
procedure if necessary we eventually reduce $w_2$ to a right-equivalent
element of the same form which is rigid at $u$.

Thus in all the cases we show that $w\er y.u$ such that $d =
y.u.y^{-1}\in\D$. It remains to prove that $d \er y.u$ which can be
done applying the same procedure as in the previous paragraph to $d$
instead of $w$ and using the fact that $y.u$ is rigid at $u$ and
satisfies $a(y) < a(u)$ by the construction. This completes the proof
of (1) and (2).

To show (3) we just need to recall the definition of the two-sided
equivalence relation and then refer to the previous argument together
with its analogue for the left cells.
\end{proof}

Let us consider two examples which demonstrate the necessity of the
main conditions imposed in the conjectures.

\begin{exam}
Let $W$ be an affine group of type $\tilde{{\mathrm A}}_4$ with
extended Dynkin diagram labeled as in Figure~\ref{fig:a4}.

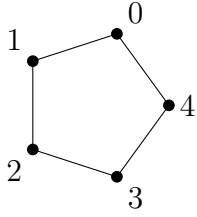
\begin{figure}[htb]
\begin{center}
\begin{tikzpicture}
\draw (1,0) -- (0.30902,0.95106) node[above right] {$0$} ;
\draw (0.30902,0.95106) -- (-0.80902,0.58779) node[above left] {$1$};
\draw (-0.80902,0.58779) -- (-0.80902,-0.58779)  node[below left] {$2$};
\draw (-0.80902,-0.58779) -- (0.30902,-0.95106)  node[below right] {$3$};
\draw (0.30902,-0.95106) -- (1.0000,0.000000000)  node[right] {$4$};
\filldraw (1,0) circle (2pt);
\filldraw (0.30902,0.95106) circle (2pt);
\filldraw (-0.80902,0.58779) circle (2pt);
\filldraw (-0.80902,-0.58779) circle (2pt);
\filldraw (0.30902,-0.95106) circle (2pt);
\end{tikzpicture}
\end{center}
\caption{Dynkin diagram for $\tilde{{\mathrm A}}_4$.\label{fig:a4}}
\end{figure}

The element $v_1 = s_4s_0s_4s_2$ is the longest element of the finite
standard parabolic subgroup $W_I$ with $I =\{ s_0, s_2, s_4\}$, so
$v_1\in\D_f$. We can check by direct computation that
$s_1\,s_4s_0s_4s_2\,s_1$, $s_3s_1\,s_4s_0s_4s_2\, s_1s_3$, and
$s_2s_3s_1\,s_4s_0s_4s_2\,s_1s_3s_2$ are in $\D$, which agrees with
Conjecture~\ref{conj1}. However, the same computation shows that
$s_0s_2s_3s_1\,s_4s_0s_4s_2\,s_1s_3s_2s_0$ is not in $\D$! The only
possible reason for this in view of the conjecture is that
$s_0s_2s_3s_1\,s_4s_0s_4s_2$ might not be rigid at $v_1$, and, indeed,
we can check that
$$s_0s_2s_3s_1\,s_4s_0s_4s_2 = s_2s_0s_3s_1\,s_0s_4s_0s_2 = s_2\,s_3s_0s_1s_0\,s_4s_0s_2,$$
where $s_3s_0s_1s_0\in\D_f$ and $a(s_3s_0s_1s_0) = a(v_1)$.
\end{exam}

\begin{exam}
It is easy to check that if $W_I$ is a finite standard parabolic
subgroup of $W$ then its longest element $w_0$ is always in
$\D_f$. The statements of our conjectures would simplify quite a bit
if all involutions in $\D_f$ had this form.  However, this is
not always the case. For example, let $W$ be a Weyl group of type
${\mathrm D}_4$ with Dynkin diagram as in Figure~\ref{fig:d4}.

\begin{figure}[htb]
\begin{center}
\begin{tikzpicture}
\draw (1,0) node [right] {$1$} -- (0,0) node [above right] {$2$};
\draw (-.5, 0.8660254) node [above  left] {$3$} -- (0,0);
\draw (-.5, -0.8660254) node [below left] {$4$} -- (0,0);
\filldraw (1,0) circle (2pt);
\filldraw (0,0) circle (2pt);
\filldraw (-0.5,0.8660254) circle (2pt);
\filldraw (-0.5,-0.8660254) circle (2pt);
\end{tikzpicture}
\end{center}
\caption{Dynkin diagram for ${\mathrm D}_4$.\label{fig:d4}}
\end{figure}
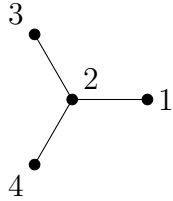
Let $v = s_2s_4s_1s_3s_2s_1s_3s_2s_4s_2s_1$. Then $a(v) = 7$ (see
\cite{J} or \cite{BG2}) and a direct computation shows that $\delta(v) = 2$,
so we have $l(v) - a(v) - 2\delta(v) = 11 - 7 - 4 = 0$, and thus
$v\in\D$. At the same time, it is clear that $v$ is not the longest
element in $W$ or in any of its standard parabolic subgroups.
\end{exam}

\section{Results}\label{sec:res} Throughout this section we assume the
positivity conjecture and boundedness of the function $a$ on $W$. These
assumptions can be slightly relaxed in some cases but can not be
completely removed. The positivity and boundedness conjectures are
widely believed to be true for any Coxeter group and are proved in a
number important special cases, we refer to the previous section for a
related discussion.

In this section we will make an extensive use of the results from \cite{L2}.
Let us emphasis once again that although \cite{L2} has a standing assumption
that the group $W$ is crystallographic, this assumption is not required for
the results that we use as it follows from the proofs in \cite{L2}.

\begin{thm}\label{thm1}
Let $v = x.v_1.x^{-1}\in\D$ with $v_1\in\D_f^\bullet$, $a(v) = a(vs)$
and $\cL(vs)\smallsetminus\cR(vs)\neq\emptyset$; and let $v' =
s.v.s$. Then if $v'$ is rigid at $v_1$, we have $v' \in \D$.
\end{thm}

\begin{proof}
The argument naturally splits into two steps:
\begin{itemize}
\item[(a)] show that $\delta(vs)$, which is $\dg(P_{e,vs})$, equals $\delta(v)$;
\item[(b)] show that $\delta(svs) = \delta(vs) + 1$.
\end{itemize}

The first step is a corollary of some results from the correspondence
of Lusztig and Springer \cite{LS}. Let us briefly recall the argument
(see also \cite[\S1.4]{Xi1}).

We have $\mu(v, vs) = 1$ (by \cite[2.3.f]{KL}) and $a(v) = a(vs)$ (by
the assumption), so by \cite[1.9]{L2}, $v \er vs$. As $\cL(vs)\neq
\cR(vs)$, clearly, $vs \not\er sv = (vs)^{-1}$. By Springer's formula
$\mu(v,vs)$, which equals $ \mu(vs,v)$, can be written as
$$\mu(v,vs) =\sum_{v'\in\D} \delta_{sv,v,v'} +
\sum_{f\in\D_1}\gamma_{sv,v,f}\pi(f).$$ By \cite[Thm.~1.8]{L2}, we
have $\gamma_{sv,v,f} = \gamma_{f,sv,v}$. If $\gamma_{f,sv,v} \neq 0$,
then \cite[Prop.~1.4(a)]{L2} implies that $f = sv$ (and
$\gamma_{f,sv,v} = 1)$. Thus in this case $sv$ and also $vs$ are in
$\D_1$.

Assume now that $vs$ is not in $\D_1$. Then by the previous argument
all $\gamma_{sv,v,f} = 0$ and hence $\mu(v,vs) = \sum_{v'\in\D}
\delta_{x^{-1},v,v'}$. But if $\delta_{x^{-1},v,v'}\neq 0$, then by a
result of Springer we have $sv\er v'$ and $v \el v'$. Since each left
cell contains only one distinguished involution (\cite[Thm.~1.10]{L2})
we must have $v = v'$. Hence we get $sv \er v$. But
$\cL(sv)\neq\cL(v)$ and we come to a contradiction with
\cite[Prop.~2.4(ii)]{KL}.

Therefore, we must have $vs\in \D_1$ which means
$$\delta(vs) = \frac12(l(vs) - a(vs) - 1) = \frac12(l(v) - a(v) - 1) = \delta(v).$$


We proceed with the second step.


By \cite[2.2.c]{KL}, we have
$$P_{e,svs} = qP_{s,vs} + P_{e,vs} - \sum_{\substack{z \prec vs \\ sz
< z}} \mu(z,vs)q_z^{-1/2}q_{vs}^{1/2}q^{1/2}P_{e,z} = qP_{s,vs} +
P_{e,vs} - \Sigma;$$ and by \cite[2.3.g]{KL}, $P_{s,vs} = P_{e,vs}$,
thus
$$P_{e,svs} = (q+1)P_{e,vs} - \Sigma.$$ As $z\prec vs$ and $s \in
\cL(z)\smallsetminus\cL(vs)$, we have $z\le_L vs$ for any $z$ in
$\Sigma$. By \cite[1.5(c)]{L2}, this implies $a(z) \ge a(vs) =
a(v)$. Now, by \cite[1.3(a)]{L2}, we obtain
$$\delta(z) \le \frac12(l(z) - a(z)) \le \frac12(l(z) - a(v)).$$ This
inequality enables us to bound the degree $d_1$ of the summands in
$\Sigma$:
$$ d_1 \le \frac12 (-l(z) + l(vs) + 1 + l(z) - a(z)) \le \frac12
(2\delta(v) + a(v) + 1 + 1 - a(v)) = \delta(v)+1,$$ thus by Step~(a),
$d_1 \le \delta(vs) + 1$. Note that the leading term of
$(q+1)P_{e,vs}$ may cancel only if $d_1 = \delta(vs) + 1$, in which
case we obtain
$$a(z) = a(v) \quad \text{and} \quad \delta(z) = \frac12 (l(z) -
a(z)).$$ Hence any such $z$ has to be a distinguished
involution. Moreover, as $a(z) = a(vs)$ and $z \le_L vs$, by
\cite[1.9(b)]{L2} we have $z\el vs$. Hence we have $\cL(z) = \cR(z) =
\cR(vs)$. Now $z\prec vs$ and
$\cL(vs)\smallsetminus\cR(vs)\neq\emptyset$ together with
\cite[2.3.e]{KL} imply $vs = tz$ for some $t\in S$.  
If $t\in\cL(x)$, we have $vs = x.v_1.x^{-1}.s = tz = tz^{-1} = x_1.v_1.x_2$ 
with $l(x_1) = l(x)+2$ and $l(x_2) = l(x)-1$, which contradicts the rigidity of $vs$ 
(and also $svs$) at $v_1$.
The case when $t\not\in\cL(x)$ can be treated using $(s,t)$-strings similar to Proposition~5.12 in 
\cite{Sh1}\footnote{This part of the argument was missing in the previous version.}.
We have $v = x.v_1.x^{-1} \in \langle t,s\rangle y \langle s,t\rangle$, where $\langle s,t\rangle$
denotes the subgroup of $W$ generated by $s$, $t$ and $s,t\notin \cL(y)\cup \cR(y)$. 
If $v_1$ is a subword of $y$, we come to a contradiction with the rigidity of 
$vs$ at $v_1$ as before.
Otherwise, one of the following cases takes place:
$v_1 = sys$, or $v_1=tyt$, or $v_1 = ty = yt$, $sy \ne ys$, or $v_1=sy=ys$, $ty \ne yt$
(there are no other possibilities because both $s$ and $t$ cannot be simultaneously 
involved in $v_1$ by rigidity). 
In each of the four cases, using the fact that both $v$ and $z = tvs$ are involutions,
we again quickly arrive to a contradiction with the rigidity assumption.

We showed that such $z$ does not exist.  Therefore the degree of
$P_{e,svs}$ is equal to the degree of $(q+1)P_{e,vs}$, and recalling
again the result of Step~(a), we have
$$\delta(svs) = \delta(vs) + 1 = \delta(v) + 1.$$ On the other hand,
$l(svs) = l(v)+2$ and $a(svs) \ge a(vs) = a(v)$, so $l(svs) - a(svs)
-2\delta(svs) \le 0$ but for any $w\in W$, $l(w) - a(w) -2\delta(w)
\ge 0$ (see \cite[1.3(a)]{L2}). We thus showed that $a(svs) = a(v)$,
$l(svs) - a(svs) -2\delta(svs) = 0$ and $svs\in\D$.
\end{proof}

We see that with the extra condition
$\cL(vs)\smallsetminus\cR(vs)\neq\emptyset$ and the assumption that
$v'$ is rigid at $v_1$, Conjecture~\ref{conj1} is true. The extra
condition is used only in Step~(b) of the proof.  The rigidity
assumption here obviously implies the rigidity of $sxv_1$ at $v_1$ (as
in Conjecture~\ref{conj1}). The converse is probably also true but we
do not know how to show it in a general setting.


The proof of Theorem~\ref{thm1} together with the following argument
allows us to establish also a special case of Conjecture~\ref{conj2}:

\begin{thm}\label{thm2}
Let $w = x.v_0 = t_n\ldots t_1.s_l\ldots s_1$ with $t_i, s_i \in S$,
$v_0 = s_l\ldots s_1\in\D_f$ is the longest element of a standard
finite parabolic subgroup of $W$ which is maximal in $w$ and $a(w) =
a(v_0)$; $u = y.u_0.y^{-1}\in\D$ with $u_0\in\D_f$ such that $a(u) =
a(u_0) = l$; and $w' = w.u.v_{01}$ with $v_{01}=s_1\ldots s_{l-1}$ has
$a(w')=a(w)$ and $\cR(w')\subsetneq\cR(w)$.

Assume that
\begin{itemize}
\item[(1)] For any $v_j = t_{j}\ldots t_1 v_0 t_1\ldots t_{j}$, $j =
0,\ldots, n-1$ and $t=t_{j+1}$ or $t = t_{j-1}$ if
$t_{j-1}\not\in\cR(v_j)$, we have $a(v_jt) = a(v_j)$,
$\cL(v_jt)\smallsetminus\cR(v_jt)\neq\emptyset$ and $tv_jt$ is rigid
at $v_0$.
\item[(2)] For any $u_j = s_{j-1}\ldots s_1 u s_1\ldots s_{j-1}$, $j = 1,\ldots, l-1$ with $u_1 = u$, we have $a(u_js_{j}) = a(u_j)$, $\cL(u_js_{j})\smallsetminus\cR(u_js_{j})\neq\emptyset$ and $s_{j}u_js_{j}$ is rigid at $u_0$;
\end{itemize}
Then $\mu(w,w')\neq 0$ and $w \er w'$.

\end{thm}

We note that conditions (1) and (2) are imposed in order to apply
Theorem~\ref{thm1}. If we would be able to prove Conjecture
\ref{conj1} in its full generality, the assumptions of the theorem
would immediately simplify.

\begin{proof}
Let us first consider the case when $x = e$, so $w = v_0$.  As $v_0 =
s_l\ldots s_1$ is the longest element of a standard parabolic
subgroup, for every $i$, $s_i\in\cL(v_0) = \cR(v_0)$. Thus by \cite[2.3.g]{KL},
$$ P_{v_0,\ v_0uv_{01}} = P_{e,\ v_0uv_{01}}.$$ Theorem~\ref{thm1}
applied to each $u_j$, $j = 1, \ldots, l-1$ as in (2) implies that
$u_{l-1} = v_{01}^{-1}uv_{01} \in \D$, moreover, the argument of Step
(a) of the proof shows that $\dg(P_{e,\ v_0uv_{01}}) = \dg(P_{e,\
v_{01}^{-1}uv_{01}}).$ So
$$\dg(P_{v_0,\ v_0uv_{01}}) = \frac12(2l(v_{01})+l(u)-a(u)) = \frac12
(a(u) +l(u) - 2) = \frac12 (l(u) + l(v_{01}) - 1).$$ Hence $\mu(w, w')
\neq 0$ (in fact, we have $\mu = 1$). Now, as
$\cR(w')\subsetneq\cR(w)$, we have $w\le_R w'$ by the definition of
the preorder $\le_R$ (see \cite{KL}). The opposite inequality $w'\le_R
w$ is easy to show using induction by $l(w')-l(w)$ and relations of
the form $w_is \le_R w_i$ which follow from the definition. Thus we
obtain that $w \er w'$.

Now let $w = x.v_0$ with $x = t_n\ldots t_1$ nontrivial. In order to
prove the theorem for this case we use induction on
the length of $x$. The base case $x = e$ has already been considered. Assume
that the theorem is proven for all $w_i = t_i\ldots t_1.v_0$, $i = 0,
\ldots, n-1$ and corresponding $w'_i$. We need to show that then it
follows for $w_n = s.xv_0$ and $w'_n = s.xv_0uv_{01}$ with $s = t_n$,
in particular, the assumption (1) is satisfied for $s$.

By \cite[2.2.c]{KL}, we have
\begin{eqnarray*}
P_{sxv_0,sxv_0uv_{01}} & = & P_{w_{n-1},w'_{n-1}} + qP_{w_n,w'_{n-1}} - \sum_{\substack{w_n\le z \prec w'_{n-1} \\ sz < z}} \mu(z,w'_{n-1})q_z^{-1/2}q_{w'_{n-1}}^{1/2}q^{1/2}P_{w_n,z} \\
& = & P_{w_{n-1},w'_{n-1}} + qP_{w_n,w'_{n-1}} - \Sigma.
\end{eqnarray*}
By the induction hypothesis, $\mu(w_{n-1},w'_{n-1})\neq 0$. It follows
that a summand of $\Sigma$ can have the same degree as
$P_{w_{n-1},w'_{n-1}}$ only if it corresponds to $z$ with $w_n \prec
z$.

Thus we have to consider $z\in W$ such that
$$ sxv_0 \prec z \prec xv_0uv_{01}.$$ By assumption (1), there
exists $t\in S$ such that $t\in \cL(xv_0x^{-1}s)\smallsetminus
\cR(xv_0x^{-1}s) = \cL(xv_0)\smallsetminus \cL(sxv_0)$. It then
belongs to $\cL(xv_0uv_{01})\smallsetminus \cL(sxv_0)$. We consider two
possible cases.
\begin{itemize}
\item[(1)] First assume $t\not\in \cL(z)$. Then by \cite[2.3.e]{KL} we have $tz =
w_{n-1}' = tx'v_0uv_{01}$ and thus $z = w_{n-2}'$. We have $z =
w_{n-2}' \er w_{n-2}$ by the induction hypothesis (assuming $n\ge2$),
which is then easily seen to be right equivalent to $x'v_0{x'}^{-1}$
using the definition of $\er$.  Theorem~\ref{thm1} then implies
$x'v_0{x'}^{-1}\in \D$.  On the other hand, by a similar argument the
relations $sxv_0\prec z$, $\cR(z)\subsetneq\cR(sxv_0)$ and $a(z) =
a(sxv_0)$ imply $z \er sxv_0 \er sxv_0x^{-1}s \in \D$. We thus see
that both $x'v_0{x'}^{-1} \er sxv_0x^{-1}s$ are in $\D$ and are not
equal to each other, which is a contradiction.

The case $n = 1$ has to be considered separately.  We get $tz =
v_0uv_{01}$, which implies $\cL(z) \subset \{s_1,\ldots,s_l\} =
\cL(v_0)$ but $s\not\in\cL(v_0)$ and thus we arrive at a contradiction
again.
\item[(2)] Now assume $t \in \cL(z)$. As $t\not\in \cL(sxv_0)$, by
\cite[2.3.e]{KL} we have $z = t.sxv_0$. A similar argument to that
given above shows that $z \er xv_0uv_{01} \er xv_0$, which gives rise
to two different but equivalent distinguished involutions --- a
contradiction.
\end{itemize}
It follows that such $z$ does not exist.  Hence $\deg(P_{w_n,w'_n}) =
\deg(P_{w_{n-1},w'_{n-1}})$, $\mu(w_n,w'_n)\neq 0$ and $w_n\er
w'_n$. This finishes the proof of the theorem.
\end{proof}

\section{Applications}\label{sec:appl}

\subsection{Groups of type $(n)$}

Let $W$ be a Coxeter group whose Coxeter matrix non-diagonal entries
$m_{i,j} = m(s_i, s_j)$, $i\neq j$ are either $n$ or infinity. Thus the
only non-trivial relations in $W$ are those of the form $(s_is_j)^n =
e$ with $s_i,s_j\in S$. We say that such Coxeter groups are {\em of
type $(n)$}.

Coxeter groups of type $(n)$ include all dihedral groups, the affine
Weyl groups $\tilde{\mathrm{A}}_{1}$ and $\tilde{\mathrm{A}_{2}}$, and
infinitely many hyperbolic Coxeter groups. We claim that if some
elements of an infinite group of type $(n)$ satisfy the conditions of
Conjectures~\ref{conj1} and \ref{conj2}, then they also satisfy the
conditions of Theorems~\ref{thm1} and \ref{thm2}. This follows easily
from the combinatorics of the relations in $W$ and the fact that the
only finite standard parabolic subgroups of $W$ of rank $>1$ have the
form $W_I$ with $I = \{s_i, s_j\}$, and all such subgroups are
isomorphic to each other.  Hypothesis~\ref{conj_a} for groups of type
$(n)$ can be verified using Tits' elementary $M$-operations similar to
the proof of Theorem~4.2 in \cite{Bel}. Although the positivity
conjecture is not known in general for such groups there are partial
results in its direction (see \S2 for a short discussion).

Let us consider some concrete examples.

\begin{exam}
Let $W = \langle s_1, s_2, s_3 \mid s_1^2,\ s_2^2,\ s_3^2,\
(s_1s_2)^3,\ (s_2s_3)^3,\ (s_3s_1)^3\rangle$ be the affine Weyl group
$\tilde{{\mathrm A}}_2$, which is of type $(3)$. Then $\D_f = \{s_1,\ s_2,\ s_3,\
s_1s_2s_1,\ s_2s_3s_2,\ s_3s_1s_3\}$ and $\D_f^\bullet = \{s_1s_2s_1,\
s_2s_3s_2,\ s_3s_1s_3\}$. Let $v_1 = s_1s_2s_1 \in \D_f^\bullet$. By
Conjecture~\ref{conj1} (or here Theorem~\ref{thm1}) it gives rise to
the distinguished involution $s_3\,s_1s_2s_1\,s_3$ in $W$. If we try
to continue the process, we see that both reduced elements
$s_2s_3\,s_1s_2s_1\,s_3s_2$ and $s_1s_3\,s_1s_2s_1\,s_3s_1$ are not
rigid at $v_1$. Thus the inductive construction terminates here. The
same can be done for the other two elements in $\D_f^\bullet$. We
obtain that in total there are $6 + 3$ distinguished involutions in
$W$. Thus there are $9$ non-trivial left cells in $W$ which agrees
with the result of \cite[\S11]{L1}. The partition of $W$ into cells
can be recovered using Conjecture~\ref{conj2} which in this case is
equivalent to Theorem~\ref{thm2}. It is a simple exercise to check
that the resulting partition coincides with the one described by
Lusztig (loc.~cit.).
\end{exam}

\begin{exam}
Let $W = P_n = \langle s_1, s_2,\dots,s_n \mid s_i^2 = 1,\
(s_is_{i+1})^2 = 1,\ i = 1,\ldots, n\rangle$ (for the sake of
convenience we assume $s_{n+1} = s_1$).  Assume $n\ge 5$.  Then $W$
can be realized as a group generated by reflections in the hyperbolic
plane about the sides of a right-angled $n$-gon. This is a special
class of right-angled Coxeter groups studied in detail in \cite{Bel}
and is a group of type $(2)$. We have $\D_f = \{ s_i, s_is_{i+1} \mid
i = 1,\ldots, n \}$ and $\D_f^{\bullet} = \{ s_is_{i+1} \mid i =
1,\ldots, n \}$.  By Theorem~\ref{thm1}, the elements of the form
$t_1\ldots t_k s_is_{i+1} t_k \ldots t_1$ with $m(t_j, t_{j+1}) =
\infty$ ($j = 1,\ldots, k-1$) and $m(t_k,s_i) = m(t_k,s_{i+1}) =
\infty$ are in $\D$. These are, in fact, the only elements of $W$
that satisfy the conditions of the theorem, and thus from the
discussion above we conclude that the elements of this form give us
the whole set $\D$ of the distinguished involutions. By Theorem
\ref{thm2} we obtain that together with the obvious right equivalences
$W$ also admits the equivalences of the form $xs_is_{i+1}\er
xs_is_{i+1}t_1\ldots t_k s_js_{j+1}t_k\ldots
t_1s_i$. Theorem~\ref{prop1} now shows that these equivalences are
sufficient for describing the right cells of $W$.

The resulting description of the right cells and distinguished
involutions in $W$ agrees with the one which is given in \cite{Bel},
so we obtain an alternative proof for the main results there.
\end{exam}

\subsection{Groups with infinitely many one-sided cells} Consider Coxeter groups $W$ with the following property:
\begin{itemize}
\item[(*)] There exist $s, t_1, t_2\in S$ such that $m(s,t_1) =
m(s,t_2) = \infty$ and $m(t_1,t_2)$ is finite.
\end{itemize}
We claim that, assuming the positivity and boundedness conjectures, (*)
implies that $W$ has infinitely many one-sided cells. Indeed, let
$v_0$ be the longest element in the standard parabolic subgroup
generated by $t_1$ and $t_2$. Then all elements with reduced
expressions of the form $t_is\ldots t_js v_0 st_j \ldots st_i$ ($i,j =
1$ or $2$) satisfy the conditions of Theorem~\ref{thm1}, and hence are
in $\D$. As under our assumptions each one-sided cell of $W$ contains
a uniquely defined distinguished involution (\cite{L2}), the claim
follows.

Let us put this result in a general perspective. The coexistence of
infinite and finite exponents in the Coxeter matrix of $W$ implies
that $W$ contains a non-abelian free subgroup (in our case it is the
subgroup generated by, for example, $t_1t_2s$ and $t_1st_2$). A group
is called {\em large} if a subgroup of finite index in it has a
non-abelian free quotient. In \cite{MV} it was shown that any
non-affine infinite indecomposable Coxeter group of finite rank is
large. Any large group is SQ-universal, and hence contains a
non-abelian free subgroup. Thus an indecomposable infinite Coxeter
group is either affine or contains a non-abelian free subgroup. By the
work of Lusztig \cite{L2}, any affine Coxeter group has only finitely
many cells. Our result here gives a support to the {\em conjecture}
that all other indecomposable infinite Coxeter groups except those
whose all exponents are infinite should have infinitely many one-sided
cells.

Let us point out that there are large groups for which one cannot
produce infinitely many distinguished involution using only Theorem
\ref{thm1}.  For a simple example consider the Hurwitz triangle group
$(2,3,7)$ with presentation $\Gamma = \langle s_1, s_2, s_3 \mid
(s_1s_3)^2 = (s_1s_2)^3 = (s_2s_3)^7\rangle$.
It is easy to check that applying Theorem~\ref{thm1} while starting
from the distinguished involutions in the standard parabolic subgroups
of $\Gamma$ we can only get finitely many elements of $\D$, while
Conjecture~\ref{conj1}, if true, implies that it should be possible to
continue the process to infinity. In \cite{BG2} we present an
experimental conformation of the infiniteness of the number of
one-sided cells as well as our main conjectures for this and some
other groups.


\begin{thebibliography}{xxx}

\bibitem[Bed]{Bed} R. B{\'e}dard, Left {V}-cells for hyperbolic {C}oxeter groups, {\em Comm. Algebra}, {\bf 17} (1989), 2971--2997.

\bibitem[Bel]{Bel} M. V. Belolipetsky, Cells and representations of right-angled Coxeter groups, {\em Selecta Math., N.~S.} {\bf 10} (2004), 325--339.

\bibitem[BG1]{BG1} M. V. Belolipetsky, P. E. Gunnells, {K}azhdan-{L}usztig cells in infinite {C}oxeter groups, Proceedings of the Workshop on 'Algebra, Combinatorics and Dynamics', {\em to appear}.

\bibitem[BG2]{BG2} M. V. Belolipetsky, P. E. Gunnells, Cells in {C}oxeter groups, II, in preparation.

\bibitem[G]{G} P. E. Gunnells, Cells in {C}oxeter groups, {\em Notices of the AMS}, {\bf 53} (2006), 528--535.


\bibitem[J]{J} J. Du, Cells in the Affine Weyl groups of type $\tilde{\mathrm D}_4$, {\em J. Algebra}, {\bf 128} (1990), 384--404.

\bibitem[KL1]{KL} D. Kazhdan, G. Lusztig, Representations of {C}oxeter groups and {H}ecke algebras, {\em Invent. Math.} {\bf 53} (1979), 165--184.

\bibitem[KL2]{KL2} D. Kazhdan, G. Lusztig, {S}chubert varieties and {P}oincar\'e duality, in {\em Proc. Sympos. Pure Math.}, XXXVI, Amer.
Math. Soc., Providence, RI, 1980, pp. 185--203.

\bibitem[Ku]{Ku} S. Kumar, {\em Kac-Moody Groups, their Flag Varieties and Representation Theory}, Progr. Math., {\bf 204}, Birkh\"auser, Boston, 2002.

\bibitem[Li]{Li} N. Libedinsky, New bases of some {H}ecke algebras via {S}oergel bimodules, 
    {\em Adv. Math.}, {\bf 228} (2011), 1043--1067.

\bibitem[L1]{L1} G. Lusztig, Cells in affine {W}eyl groups, {\em Algebraic Groups and Related Topics,} Adv. Stud. Pure Math., {\bf 40}, Kinokuniya and North-Holland, Amsterdam, 1985, pp.~255--287.

\bibitem[L2]{L2} G. Lusztig, Cells in affine {W}eyl groups, II, {\em J. Algebra,} {\bf 109} (1987), 536--548.

\bibitem[L3]{L3} G. Luzstig, {\em Hecke algebras with unequal parameters,} CRM Monographs Ser., {\bf 18}, Amer. Math. Soc., 2003.

\bibitem[LS]{LS} G. Lusztig, T. A. Springer, Correspondence, 1987.

\bibitem[MV]{MV} G. A. Margulis, {\`E}. B. Vinberg, Some linear groups virtually having a free quotient, {\em J. Lie Theory,} {\bf 10} (2000), 171--180.


\bibitem[Sh1]{Sh1} J. Y. Shi, The joint relations and the set $\mathcal D\sb 1$ in certain crystallographic groups, {\em Adv. Math.}, {\bf 81} (1990), 66--89.

\bibitem[Sh2]{Sh2} J. Y. Shi, A counterexample to a conjecture of Lusztig, {\em J. Algebra}, {\bf 323} (2010), 2591--2598.

\bibitem[So]{So} W. Soergel, {K}azhdan-{L}usztig polynomials and indecomposable bimodules over polynomial rings,
{\em Journal of the Inst. of Math. Jussieu}, {\bf 6} (2007), 501--525.

\bibitem[Xi1]{Xi1} N.~Xi, On the characterization of the set $\mathcal D\sb 1$ of the affine {W}eyl group of type $\tilde{A}\sb n$, {\em Representation theory of algebraic groups and quantum groups,} Adv. Stud. Pure Math., {\bf 40}, Math. Soc. Japan, Tokyo, 2004, pp. 483--490.

\bibitem[Xi2]{Xi2} N.~Xi, Lusztig's $a$-function for Coxeter groups with complete graphs,
{\em Bull. Inst. Math. Acad. Sin. (N.S.)}, {\bf 7} (2012), 71--90.

\bibitem[Zh]{Zh} P.~Zhou, Lusztig's $a$-function for Coxeter groups of rank $3$, {\em preprint}, arXiv:1107.1995v1 [math.QA].

\end{thebibliography}
\end{document}